\documentclass[12pt]{amsart}
\usepackage{amsmath, amsthm, amscd, amsfonts, amssymb, latexsym, graphicx, color}

\usepackage[utf8]{inputenc}

\usepackage{amsmath,amsthm}
\usepackage{amssymb,latexsym}
\usepackage{enumerate}

\newtheorem{theorem}{Theorem}[section]
\newtheorem{lemma}[theorem]{Lemma}
\newtheorem{proposition}[theorem]{Proposition}

\theoremstyle{definition}

\theoremstyle{remark}

\numberwithin{equation}{section}

\begin{document}

\title[About closed definable subsets]{About closed subsets definable \\ in Hensel minimal structures}

\author[Krzysztof Jan Nowak]{Krzysztof Jan Nowak}




\date{}

\vspace{3ex}

\begin{abstract} The main purpose is to establish two theorems about closed 0-definable subsets $A$ of an affine space $K^{n}$ over a Hensel minimal field $K$. The first, being
 a non-Archimedean counterpart of one from o-minimal geometry, states that every such subset $A$ is the zero locus of a continuous 0-definable function on $K^{n}$. The second is a definable, non-Archimedean version of the Tietze--Urysohn extension theorem. The proofs use ubiquity of clopen sets in non-Archimedean geometry and a description of definable sets.
\end{abstract}

\maketitle


\vspace{5ex}

\section{Introduction}

\vspace{3ex}

This paper is a continuation of our research on Henselian valued fields from articles~\cite{K-N,Now-Sel,Now-Sym,Now-Sing,Now-Alant,Now-strat} and on Hensel minimal structures from~\cite{Now-tame,Now-Lip}.

\vspace{1ex}

We are concerned with Hensel minimal structures on non-trivially valued fields $K$ of equicharacteristic zero. And more precisely, with models $K$ of a 1-h-minimal theory in an expansion $\mathcal{L}$ of the language of valued fields.
The main purpose is to establish the following two results concerning closed 0-definable subsets $A$ of an affine space $K^{n}$.

\begin{theorem}\label{zero-set} (On zero locus)
Every closed 0-definable subset $A$ of $K^n$ is the zero locus $\mathcal{Z}(g) := \{ x \in K^{n}: \ g(x)=0 \}$ of a continuous 0-definable function $g$ on $K^n$.
\end{theorem}


\begin{theorem}\label{ext} (On definable extension)
Let $A$ be a closed 0-definable subset of $K^{n}$. Then
every continuous 0-definable function $f: A \to K$ can be extended to a continuous 0-definable function $F: X \to K$.
\end{theorem}

The first result is a counterpart of one from o-minimal geometry~(see e.g.~\cite[Proposition~2.7.5]{BCR} and \cite[Property~4.22]{Dries-M}); the second is a definable non-Archimedean version of the Tietze--Urysohn extension theorem. The proofs of these theorems make use of a model-theoretic compactness argument and ubiquity of clopen sets in non-Archimedean geometry.

\vspace{1ex}

Note that the axiomatic theory of Hensel minimal structures was introduced by Cluckers--Halupczok--Rideau~\cite{C-H-R}, whose research followed numerous earlier attempts to find suitable approaches in geometry of valued fields, which would realize, like o-minimality in real geometry, the postulates of both tame topology and tame model theory. Those attempts led to various, axiomatically based concepts, including C-minimality~\cite{Has-Macph-1,Macph}, P-minimality~\cite{Has-Macph-2}, V-minimality~\cite{H-K}, b-minimality~\cite{C-L}, tame structures~\cite{C-Com-L,C-Fo-L}, and eventually Hensel minimality.

\vspace{1ex}

In the purely topological case, non-Archimedean extension problems were investigated by Ellis (cf.~\cite{El-1,El-2}). He studied the extension of continuous maps defined on closed subsets of zero-dimensional spaces $X$ with values in various types of metric spaces; in particular, continuous maps from ultraparacompact (ultranormal) spaces $X$ into a (separable) complete metric space. He obtained, among others, a non-Archimedean analogue of the Tietze--Urysohn theorem on extension of continuous functions from a closed subset of an ultranormal space $X$ into a locally compact field with non-Archimedean absolute value.

\vspace{1ex}

Often the problem of extension can be related to that of retractions. The existence of a continuous retraction on any closed subset $A$ of an ultrametric space $X$ was first proven in~\cite{Da} using some of Ellis' ideas~\cite{El-2}, and in~\cite{BDHM} by means of a well-order on $X$. What was even proven in the latter paper was the existence of a Lipschitz retraction with any Lipschitz constant $>1$. More recently, the topic of non-Archimedean extensions and retractions was discussed in~\cite{K-K-S}.

\vspace{1ex}

In the definable, non-Archimedean and non-locally compact settings, however, even the existence of definable continuous retractions onto arbitrary closed subsets seems to be less plausible due to the lack of Skolem functions. In particular, there exist 0-definable open balls without any 0-definable point.

\vspace{1ex}

In Hensel minimal settings, a definable, non-Archimedean and non-locally compact analogue of Kirszbraun's theorem on Lipschitz extension, with an arbitrarily small magnification of the Lipschitz constant, was established in our paper~\cite{Now-Lip} . A key role in the proof was played by our concept of a Lipschitz package with skeleton.

\vspace{1ex}

In the definable $p$-adic, thus locally compact settings, the existence of Lipschitz retractions, with constant 1, and Kirszbraun's theorem on Lipschitz extension, with the Lipschitz constant preserved, were proven in~\cite{C-M} via simultaneous induction and using Skolem funtions, which are available in this case (cf~\cite{Dries}).

\vspace{1ex}

In the purely topological case, the following ultrametric version of Kirsz\-braun's extension theorem was established (using of the Zorn's lemma) by Bhaskaran~\cite[Theorems~1.2 and~1.5]{Bas}:

\vspace{1ex}

\begin{em}
Let $K$ be a rank one valued field, $X$ be an ultrametric space, and $A \subset X$. Then every bounded Lipschitz function $f:A \to K$ extends to a bounded Lipschitz function $F:X \to K$, with the Lipschitz constant and supremum norm preserved, whenever the field $K$ or the subspace $A$ is spherically complete.
\end{em}

\vspace{1ex}

Moreover, he proved that the field $K$ is spherically complete if and only if the above extension theorem holds for every ultrametric space $X$ and every $A \subset X$. It is thus not always possible to extend Lipschitz functions to the whole ultrametric space with the same Lipschitz constant. Hence there exist closed subsets $A$ of an ultrametric space $X$ that are not Lipschitz retracts of $X$ with constant 1.

\vspace{1ex}

In the Hensel minimal settings, we proved in~\cite{Now-tame} that every definable locally closed subset of an affine space $K^{n}$ and, more generally, every Hausdorff definable LC-space (a concept including definable manifolds) is definably ultranormal and definably ultraparacompact.

\vspace{1ex}

Finally, let us mention some intricacies behind ultrametrics. A space $X$ is ultrametrizable if and only if (cf.~\cite{Groot}) it is both metrizable and ultranormal (equivalently, normal of large inductive dimension zero). Hence a separable metrizable space $X$ is ultrametrizable whenever it is ultraregular (equivalently, regular of small inductive dimension zero), because the small and large dimensions coincide for Lindel\"{o}f spaces (\cite[Theorem~7.1.11]{Eng}) and, a fortiori, for separable metric spaces.
Yet metrizability and ultraregularity do not entail ultrametrizability in general, because non-separable metric spaces may have discrepancy between the small and large  dimensions, as shown by Roy~\cite{Roy,Roy-2}.

\vspace{1ex}

\section{Preparatory results}

We shall adopt the multiplicative convention $|\cdot|$ for the valuation of the ground field $K$, and put
$$|x| := \max \, \{ |x_{1}|, \ldots, |x_{n}| \}, \ \ x \in K^{n}. $$

\vspace{1ex}

Suppose that $A$ is the union of a finite number of 0-definable subsets $A_{i}$ of $K^{n}$. Clearly, Theorem~\ref{zero-set} holds for $A$ if it holds for every set $A_{i}$. To achieve such a partition property for the extension theorem, we need the following result on separation of closed definable sets.

\begin{proposition}\label{separate}
Let $A_{1}$ and $A_{2}$ be closed 0-definable subsets of $K^{n}$. Then there are open 0-definable subsets $U_{1}$ and $U_{2}$ of $K^{n}$ such that

1) $A_{1} \setminus A_{2} \subset U_{1}$ and $A_{2} \setminus A_{1} \subset U_{2}$;

2) $U_{1}$ and $U_{2}$ are clopen subsets of $K^{n} \setminus (A_{1} \cap A_{2})$;

3) $U_{1} \cup (A_{1} \cap A_{2})$ and $U_{2} \cup (A_{1} \cap A_{2})$ are closed subsets of $K^{n}$.
\end{proposition}

\begin{proof}
It is not difficult to show that the sets
$$ U_{1} := \{ x \in K^{n}: \, \exists \, a_{1} \in A_{1} \setminus A_{2} \ \forall \, a_{2} \in A_{2} \ \; |x-a_{1}| < |x-a_{2}| \, \} $$
and
$$ U_{2} := \{ x \in K^{n}: \, \exists \, a_{2} \in A_{2} \setminus A_{1} \ \forall \, a_{1} \in A_{1} \ \; |x-a_{2}| < |x-a_{1}| \, \} $$
satisfy the conclusion.
\end{proof}

\begin{proposition}\label{partition} (Partition Property)
Let $A_{i}$, $i=1,\ldots,s$, be closed 0-definable subsets of $K^{n}$. Then Theorem~\ref{ext} holds for the union
$$ A := \bigcup_{i=1}^{s} A_{i} $$
if it holds for every subset $A_{i}$.    
\end{proposition}

\begin{proof}
By induction, it is enough to consider the case $s=2$. Let $f: A_{1} \cup A_{2} \to K$ be a continuous 0-definable function. By the assumption, we can assume that $f$ vanishes on $A_{2}$. Consider a continuous 0-definable extension $F_{1}: K^{n} \to K$ of the restriction of $f$ to $A_{1}$ and keep the notation of Proposition~\ref{separate}. Then the function
$$ F(x) = \left\{ \begin{array}{cl}
                        F_1 (x) & \mbox{ if } \ x \in U_{1}, \\
                        0 & \mbox{ if } \ x \in K^{n} \setminus U_{1},
                        \end{array}
               \right.
$$
is an extension of $f$ we are looking for.
\end{proof}

We shall still need an elementary lemma.

\begin{lemma}\label{lem-inter}
There is a continuous 0-definable function $g_{n}: K_{x}^{n} \to K$ such that
$$ g_{n}(x) = 0 \ \ \Longleftrightarrow \ \ x=0. $$
\end{lemma}

\begin{proof}
We define the functions $g_{n}$ inductively. Put
$$ g_{2}(x_{1},x_{2}) := \left\{ \begin{array}{cl}
                        x_1 & \mbox{ if } \ |x_2| \leq |x_1|, \\
                        x_2 & \mbox{ if } \ |x_1| < |x_2|,
                        \end{array}
               \right.
$$
and
$$ g_{k+1}(x_{1},\ldots,x_{k+1}) := g_{2}(g_{k}(x_{1},\ldots,x_{k}),x_{k+1}). $$
\end{proof}

The proofs of the main theorems will be proven by induction with respect to the dimension $k = \dim A$. The case $k=0$ is straightforward by Lemma~\ref{lem-inter}.

\vspace{1ex}

In order to prove Theorem~\ref{zero-set}, we can and shall additionally assume that the subset $A$ is bounded. Indeed, put
$$ A = \bigcup \, \{ A_{I}: \ I \subset \{ 1,\ldots,n \} \} \ \ \text{with} \ \  A_{I} := A \cap B_{I}, $$
where
$$ B_{I} := \{ x \in K^{n}: \ |x_{i}| \leq 1 \ \ \text{for} \ i \in I \ \ \text{and} \ \ |x_{j}| > 1 \ \ \text{for} \ j \not \in I \}; $$
and let
$$ \phi_{I}: B_{I} \to K^{n}, \ \ \ \phi_{I}(x) = y, $$
where $y_{i} = x_{i}$ for $i \in I$ and $y_{j} = 1/x_{j}$ for $j \not \in I$.
Then $B_{I}$ are clopen subsets of $K^{n}$, $\phi_{I}$ are homeomorphisms of $B_{I}$ onto the bounded subsets $D_{I} := \phi_{I}(B_{I})$, and $E_{I} := \phi_{I}(A_{I})$ are closed subsets of $D_{I}$. Further, the (topological) closures $\overline{E_{I}}$ of the sets $E_{I}$ are bounded subsets of $K^{n}$. Therefore, if $\overline{E_{I}}$ is the zero locus of a continuous 0-definable function $g_{I}: K^{n} \to K$, then $A_{I}$ is the zero locus of the function $f_{I} := g_{I} \circ \phi_{I}$ which is a continuous 0-definable function on the clopen subset $B_{I}$. This justifies the reduction.

\vspace{1ex}

\section{A description of definable sets}

In this section, we shall divide the set A into closed 0-definable pieces appropriately prepared for the proofs of the main results. Hence and by Proposition~\ref{partition}, the problem reduces, via throwing away pieces of lower dimension and treating them by induction, to a single prepared piece with a suitable description. This will be done using parametrized cell decomposition and a model-theoretic compactness argument.

\vspace{1ex}

So assume that $A$ is a closed 0-definable subset of $K^{n}$ of dimension $k$ with $0<k<n$. For coordinates $x = (x_1,\ldots,x_n)$ in the affine space $K^{n}$, write
$$ x=(y,z), \ \ y=(x_1,\ldots,x_k), \ \ z=(x_{k+1},\ldots,x_n). $$
Let $\pi: K^n \to K^k$ be the projection onto the first $k$ coordinates.
For $y \in K^k$, denote by $A_y \subset K^{n-k}$ the fiber of the set $A$ over the point $y$.

\vspace{1ex}

Denote by $\overline{E}$ and $\partial E := \overline{E} \setminus E$ the closure and frontier (in the valuation topology) of a set $E$, respectively. We have the basic dimension inequality (cf.~\cite[Proposition~5.3.4]{C-H-R}):
\begin{equation}\label{front}
  \dim \partial E < \dim E.
\end{equation}

\vspace{1ex}

By  parametrized cell decomposition (cf.~\cite[Theorem~5.7.3]{C-H-R}), we can assume that $A$ is (perhaps after a permutation of variables) the closure of a parametrized 0-definable cell $C = (C_{\xi})_{\xi}$ of dimension $k$ of cell-type $(1,\ldots,1,0,\ldots,0)$, with $RV$-sort parameters $\xi$ and centers $c_{\xi}$. Since definable $RV$-unions of finite sets stay finite, the restriction of $\pi$ to $C$ has finite fibers of bounded cardinality.

\vspace{1ex}

Further, by suitable 0-definable partitioning, we can assume that $A$ is the closure $\overline{E}$ of a 0-definable subset $E$ of dimension $k$ such that all the fibers
$$ E_{y}, \ \ y \in F:= \pi(E), $$
have the same cardinality, say $s$, and the sets
$$ E_{j}(y) = \{ c_{ji}(y), \ i=1,\ldots,s_j \}, \ \ j=k+1,\ldots,n, $$
of the $j$-th coordinates of points from the fibers $E_y$ have the same cardinality, say $s_j$. Since the fibers $E_{y}$ are finite, the set $F$ is of dimension $k$, and we can, throwing away pieces of lower dimension, assume that $F$  is an open 0-definable subset of $K^{k}$.

\vspace{1ex}

Now consider the polynomials
$$ P_j (y,Z_j) := \prod_{z \in C_j(y)} \, (Z_j - z) = \prod_{i=1}^{s_{j}} \, (Z_{j} - c_{ji}(y)), \ \ y \in F, \ j=k+1,\ldots,n, $$
Then
$$ P_j (y,Z_j) = Z_j^{s_j} + b_{j,1}(y) Z_j^{s_{j-1}} + \ldots + b_{j,s_j}(y), \ \ j=k+1,\ldots,n, $$
where $b_{j,i}: F \to K$, $i=1,\ldots,s_j$, are 0-definable functions.

\vspace{1ex}

We still need the following lemma on good directions, resembling to some extent the concept of primitive elements from field theory and algebraic geometry.

\begin{lemma}\label{Lem-lin}
There exist a finite number of linear functions
$$ \lambda_{l}: K^{n-k} \to K, \ \ l=1,\ldots,p, $$
with integer coefficients such that,
for every $y \in F$, $\lambda_{l}$ is injective on the Cartesian product $\prod_{j=k+1}^{n} \, E_j(y)$ for some $l=1,\ldots,p$.
\end{lemma}

\begin{proof}
The conclusion follows by a routine model-theoretic compactness argument.
\end{proof}

Hence and by a further partitioning, we can assume that one linear function
$$ \lambda: K^{n-k} \to K $$
with integer coefficients is injective on every product
$$ \prod_{j=k+1}^{n} \, E_{j}(y), \ \ y\in F. $$
Consider now the polynomial
$$ P(y,Z) := \prod_{z \in E_y} \, (Z - \lambda(z)) = Z^s + b_1(y) Z^{s-1} + \ldots + b_s(y), $$
where $b_{j}: F \to K$ are 0-definable functions.
Then
\begin{equation}\label{descript}
   E = \{ x =(y,z) \in F \times K^{n-k}: \
\end{equation}
$$ P_{k+1} (x_1,\ldots,x_k,x_{k+1}) = \ldots = P_{n} (x_1,\ldots,x_k,x_n) = $$
$$ P(x_1,\ldots,x_k,\lambda(x_{k+1},\ldots,x_{n})) = 0 \}. $$

The sets of all points at which the functions $b_{ji}(y)$ and $b_{i}(y)$ are not continuous are 0-definable subsets of $F$ of dimension $<k$ (cf.~\cite[Theorem~5.1.1]{C-H-R}).
Again, throwing pieces of lower dimension, we can additionally assume that $b_{ji}(y)$ and $b_{i}(y)$ are continuous functions on the open subset $F$. Then $E$ is a closed subset of $F \times K^{n-k}$, and thus
\begin{equation}\label{frontier}
  \partial E \subset \partial F \times K^{n-k}.
\end{equation}

\vspace{1ex}

In this manner, we have reduced the proofs of the theorems under study to the case where $A$ is the closure of the set $E$ described above by formula~\ref{descript}.
Moreover, in the proof of the first theorem, we can assume without loss of generality, as outlined before, that the set $E$ is bounded.

\vspace{1ex}

In the sequel, we shall still need the following set
$$ U := \{ (y,z) \in F \times K^{n-k}: $$
$$ |z-c_{i}(y)| < \min \, \{ |c_{i}(y) - c_{j}(y)|: \ i,j=1,\ldots,s, \ i \neq j \},  $$
$$ \forall \, b \in \partial \, F \ |z - c_{i}(y)| < |y-b| \, \}. $$
It is not difficult to check that $U$ is a 0-definable clopen subset of $(K^{k} \setminus \partial F) \times K^{n-k}$ containing $E$, $\partial U = \partial E$, and thus $U \cup \partial E$ is a closed subset of $K^{n}$.

\section{Proof of Theorem~\ref{zero-set}}

Let $A$ be a bounded closed 0-definable subset of $K^{n}$ of dimension $k$ with $0<k<n$, and suppose the conclusion of the theorem holds for all subsets of dimension $<k$.

\vspace{1ex}

By the partition property and the description of definable sets from Section~3, we can assume that $A $ is the closure $\overline{E}$ of the set $E$ given by formula~\ref{descript}.
Since $F$ is an open subset of $K^k$, its frontier $\partial F$ is a closed subset of $K^k$ of dimension $<k$ (inequality~\ref{front}). By the induction hypothesis, $\partial F$ is the zero locus of a continuous 0-definable function $f: K^k \to K$.

\vspace{1ex}

Keeping the notation from Section~3, the functions $b_{ji}(y)$ are bounded because so are the sets $A$ and $E$. Therefore the functions
$$ f(y) \cdot b_{ji}(y) \ \ \text{and} \ \ f(y) \cdot b_{i}(y) $$
extend by zero through $\partial F$ to continuous functions on $\overline{F}$. And then they extend by zero off $F$ to continuous 0-definable functions on $K^k$.

\vspace{1ex}

We can thus regard the coefficients of the following polynomials (in the indeterminates $Z_{k+1}, \ldots,Z_{n}$ and $Z$, respectively):
$$ Q_{k+1} (y,Z_j) := f(y) \cdot P_{k+1} (y,Z_{k+1}), \ \ldots \ , \ Q_n (y,Z_n) := f(y) \cdot P_n (y,Z_n) $$
and
$$ Q (y,Z) := f(y) \cdot P (y,Z), $$
as continuous 0-definable functions on  $K^k$ vanishing off the subset $F$. Put
$$ G := \{ x \in K^n: \ Q_{k+1} (x_1,\ldots,x_k,x_{k+1}) = \ldots = Q_{n} (x_1,\ldots,x_k,x_n) = $$
$$ Q(x_1,\ldots,x_k,\lambda(x_{k+1},\ldots,x_{n})) = 0 \}. $$
Then
\begin{equation}\label{eq1}
  G \cap (F \times K^{n-k}) = E
\end{equation}
and
\begin{equation}\label{eq2}
  G \cap ((K^k \setminus F) \times K^{n-k}) = (K^k \setminus F) \times K^{n-k}.
\end{equation}
But by the induction hypothesis, $\partial E$ is the zero locus of a continuous 0-definable function $e: K^n \to K$.
Then the function
$$ \widetilde{e}(x) = \left\{ \begin{array}{cl}
                        0 & \mbox{ if } \ x \in U, \\
                        e(x) & \mbox{ if } \ x \in K^n \setminus U.
                        \end{array}
               \right.
$$
is continuous 0-definable with zero locus $U \cup \partial E$:
$$ U \cup \partial E = \{ x \in K^n: \ \widetilde{e}(x) = 0 \}. $$
Hence and by equalities~\ref{eq1} and~\ref{eq2}, we get
$$
  A = E \cap \partial E = \{ Q_{k+1} (x_1,\ldots,x_k,x_{k+1}) = \ldots = Q_{n} (x_1,\ldots,x_k,x_n) =
$$
$$  Q(x_1,\ldots,x_k,\lambda(x_{k+1},\ldots,x_{n})) = \widetilde{e}(x) =0 \} \subset K^n. $$
Now it follows immediately from Lemma~\ref{lem-inter} that $A$ is the zero locus of the continuous 0-definable function
$$ g_{n-k+2}(Q_{k+1} (x), \ldots, Q_{n} (x),Q(x), \widetilde{e}(x)), $$
as desired.

\vspace{1ex}

Finally, suppose that $A$ is of dimension $k=n$. Then $A = U \cup F$ for an open 0-definable subset $U \subset K^n$ and a closed 0-definable subset $F \subset K^n$ of dimension $<n$ such that $U$ is a clopen subset of $K^{n} \setminus F$. By the induction hypothesis, $F$ is the zero locus of a continuous 0-definable function $f: K^n \to K$. Then $A$ is the zero locus of the following continuous definable function
$$ g(x) = \left\{ \begin{array}{cl}
                        0 & \mbox{ if } \ x \in A, \\
                        f(x) & \mbox{ if } \ x \in K^n \setminus A.
                        \end{array}
               \right.
$$
This completes the proof of Theorem~\ref{zero-set}.    \hspace*{\fill} $\Box$

\vspace{1ex}

\section{Proof of Theorem~\ref{ext}}

We shall keep the notation from Section~3. Let $f: A \to K$ be a continuous 0-definable function on a closed subset $A$ of $K^{n}$ of dimension $k$ with $0 < k < n$. Suppose that the conclusion holds for the subsets of $K^n$ of dimension $ <k$.

\vspace{1ex}

Again, by the partition property (Proposition~\ref{partition}), we are reduced to the case where the set $A$ is the closure $\overline{E}$ of the set $E$ given by formula~\ref{descript}.

\vspace{1ex}

By the induction hypothesis, the restriction of $f$ to the frontier $\partial E$ extends to a continuous 0-definable function on $K^{n}$. Therefore we can assume that the function $f$ vanishes on $\partial E$. Then it is not difficult to check that the function given by the formula
$$ F(y,z) := \left\{ \begin{array}{cl}
                        f(y,c_{i}(y)) & \mbox{ if } \ (y,z) \in U, \ |z-c_{i}(y)| = \omega(y,z), \\
                        0 & \mbox{ otherwise,}
                        \end{array}
               \right.
$$
where $\omega(y,z) = \min \, \{ |z-c_{j}(y)|: \ j=1,\ldots,s \}$, is a continuous 0-definable extension of $f$ we are looking for. The uniqueness of the index $i$ in the above definition follows directly from the definition of the set $U$.

\vspace{1ex}

Finally, suppose that $A$ is of dimension $k=n$. As before, $A = U \cup F$ for an open 0-definable subset $U \subset K^n$ and a closed 0-definable subset $F \subset K^n$ of dimension $<n$ such that $U$ is a clopen subset of $K^{n} \setminus F$. Again, by the induction hypothesis, we can assume that the function $f$ vanishes on $F$.
Then the function $F: K^{n} \to K$ given by the formula
$$ F(x) := \left\{ \begin{array}{cl}
                        f(x) & \mbox{ if } \ x \in U, \\
                        0 & \mbox{ otherwise, }
                        \end{array}
               \right.
$$
is a continuous 0-definable extension of $f$ we are looking for. This finishes the proof of Theorem~\ref{ext}.      \hspace*{\fill} $\Box$

\vspace{3ex}

\begin{small}
Institute of Mathematics

Faculty of Mathematics and Computer Science

Jagiellonian University

ul.~Profesora S.\ \L{}ojasiewicza 6

30-348 Krak\'{o}w, Poland

{\em E-mail address: nowak@im.uj.edu.pl}
\end{small}

\end{document}